\documentclass[a4paper,10pt]{article}

\usepackage[utf8]{inputenc}
\usepackage{authblk}
\usepackage{amsthm,amsmath,amssymb,latexsym,dsfont}

\usepackage{enumerate}
\usepackage{tikz}
\usetikzlibrary{arrows.meta, positioning}
\usepackage[numbers]{natbib}
\usepackage{xcolor}

\usepackage{url}
\usepackage{etoolbox}
\patchcmd{\thebibliography}{\section*{\refname}}{\section*{\refname}\small}{}{}
\newtheorem{theorem}{Theorem}[section]
\newtheorem{proposition}[theorem]{Proposition}

\newtheorem{remark}[theorem]{Remark}

\newcommand{\N}{\mathbb{N}}
\newcommand{\Z}{\mathbb{Z}}

\newcommand{\R}{\mathbb{R}}

\renewcommand{\P}{\mathbb{P}}
\newcommand{\E}{\mathbb{E}}

\usepackage[margin=30mm]{geometry}

\title{Functional dependence and synchronous coupling in ergodic autoregressions}

 \author{ Paul Doukhan
(\footnote{\texttt{doukhan@cyu.fr}. CY Cergy Paris University,  AGM Mathematics, 2 Bd. Adolphe Chauvin, 95000 Cergy-Pontoise, AGM-UMR 8080, France, ORCID 0000-0001-6170-7008}\ ),

 Lionel Truquet
(\footnote{\texttt{Lionel.TRUQUET@ensai.fr}.  Univ Rennes, Ensai, CNRS, CREST -- UMR 9194, F-35000 Rennes, France, ORCID 0000-0002-0017-1438 }\ ),
 
 }

\begin{document}

\maketitle
\begin{abstract}
Functional dependence measures have become an important tool in the analysis of nonlinear time series and are typically formulated with respect to a given innovation representation of the process. This note points out that the probability space on which such representations yield the expected memory loss properties may not always coincide with the natural dynamical probability space of the model. We exhibit  classes of uniformly ergodic autoregressive processes for which the behavior of the natural innovation coupling undergoes a qualitative transition as the model parameter varies. For this family of models, this transition coincides with a change in the sign of an associated Lyapunov exponent. In particular, a positive Lyapunov exponent may prevent the forgetting of initial perturbations along trajectories driven by the same innovations, despite uniform ergodicity of the associated Markov chain. These observations highlight the importance of carefully specifying the underlying probability space when interpreting or applying functional dependence measures. \end{abstract}

\vspace{1cm}
\footnoterule
\noindent
{\sl 2020 Mathematics Subject Classification:}
Primary 60J05; secondary 37H15, 62M10.\\
\noindent
{\sl Keywords and Phrases:}
Functional dependence, Bernoulli shift representation, iterated random maps, Lyapunov exponent, synchronous coupling, uniform ergodicity.

\section{Introduction}






In time series analysis, various notions of weak dependence have been introduced to establish limit theorems and study the asymptotic behavior of statistical procedures. Among the earliest and most influential are mixing conditions; see \cite{D} for a comprehensive overview. Although mixing coefficients only depend on the distribution of the process, they are often difficult to verify in nonlinear models and may even fail to hold for rather simple autoregressive or bilinear processes; see \cite{andrews3,DMT}. This has motivated the development of alternative notions of weak dependence, frequently based on coupling constructions or covariance inequalities; see \cite{Dedecker_2007}. A great overview on all the dependence conditions for Markov chains is nicely developed in \cite{kulik}.

A particularly successful approach is the functional dependence measure introduced by \cite{WBW2005}. It is based on a causal Bernoulli shift representation
\begin{equation}
\label{BS}
X_t=H(\mathcal F_t),\qquad
\mathcal F_t=(\varepsilon_t,\varepsilon_{t-1},\ldots),
\end{equation}
where $(\varepsilon_t)_{t\in\mathbb Z}$ is a sequence of i.i.d.\ random variables. Let $\varepsilon_0'$ be an independent copy of $\varepsilon_0$, and define
\[
\mathcal F_n^*
=
(\varepsilon_n,\varepsilon_{n-1},\ldots,\varepsilon_1,
\varepsilon_0',
\varepsilon_{-1},\varepsilon_{-2},\ldots).
\]
The functional dependence coefficients are then defined by
\[
\delta_q(n)
=
\left\|
H(\mathcal F_n)-H(\mathcal F_n^*)
\right\|_q,
\qquad n\ge1.
\] 
These coefficients have proved remarkably effective for deriving limit theorems. For instance, the central limit theorem holds under the simple summability condition $\sum_{n\ge1}\delta_2(n)<\infty$.

Unlike classical mixing coefficients, however, functional dependence measures are not intrinsic properties of the distribution of the process. They are attached to a particular Bernoulli shift representation, which is generally not unique. This feature is usually harmless when a convenient representation is readily available, but it becomes more delicate for Markov chains generated by iterated random maps. The classical results in \cite{NT1978} hold under weak conditions stressed by \cite{Mokk} for autoregressive cases. The general theory of dependence of Markov chains is provided in \cite{kulik}: conditions for ergodicity, geometric ergdicity and uniform ergodicity of stationary Markov processes respectively yields $\beta-$mixing, geometric $\beta-$mixing and  geometric $\varphi-$mixing  

Consider indeed an autoregressive Markov chain of the form
\[
X_t=f_t(X_{t-1}),
\qquad
f_t(x)=F(x,\varepsilon_t),
\]
where $(f_t)_{t\in\mathbb Z}$ is a sequence of i.i.d.\ random maps. 
For simplicity of this note, we only consider random maps on the real line. Applying the functional dependence measure to the natural innovation sequence amounts to comparing two trajectories driven by the same realization of the random maps after perturbing the initial innovation only. From the viewpoint of random dynamical systems, this is precisely the natural synchronous coupling associated with the innovation sequence. When the random maps satisfy a global average contractivity condition, as in \cite{WuShao2004}, this coupling contracts exponentially fast and the natural innovation representation provides a convenient Bernoulli shift representation. Beyond this contractive setting, however, the behavior of the synchronous coupling becomes much less transparent.

The purpose of this note is to point out that the natural innovation representation may fail to reflect the ergodic behavior of the underlying Markov chain. Our main result, Theorem \ref{p1}, shows that, under a positive Lyapunov exponent and mild regularity assumptions, trajectories driven by the same innovation sequence fail to synchronize almost surely. Consequently, the natural innovation representation cannot yield sufficiently decaying functional dependence coefficients, despite uniform ergodicity of the Markov chain. The problem is illustrated with a simple smooth autoregressive model of the form
\[
F(x,\varepsilon)=af(x)+\varepsilon,
\]
where $f$ is bounded, smooth and periodic. Under mild assumptions on the noise distribution, the associated Markov chain is uniformly ergodic for every value of the parameter $a$. Nevertheless, the behavior of the natural synchronous coupling changes qualitatively with $a$: depending on the sign of a Lyapunov exponent, trajectories driven by the same innovations either forget or preserve their initial perturbation.

Our purpose is not to question the usefulness of functional dependence measures, but rather to emphasize that they should be interpreted relative to the underlying probability space on which the Bernoulli shift representation is constructed. In contrast with classical mixing coefficients, functional dependence coefficients describe the stability of a specific stochastic representation rather than an intrinsic property of the process distribution. For Markov chains, this distinction naturally connects the functional dependence approach with the theory of synchronous couplings and random dynamical systems.

The paper is organized as follows. Section~\ref{2} presents two pathological examples illustrating that perturbations of the innovation sequence need not be forgotten over time, even for simple random mappings. Section~\ref{3} then turns to smooth functional autoregressive models. After a brief discussion of Lyapunov exponents and synchronous couplings, we study in detail a family of uniformly ergodic functional AR$(1)$ models. We show that the natural synchronous coupling undergoes a qualitative transition as the model parameter varies: for small values of the parameter, initial perturbations are forgotten exponentially fast, whereas for sufficiently large values they persist indefinitely. This transition is shown to coincide with a change in the sign of the associated Lyapunov exponent. Finally, Section~\ref{4} shows that this phenomenon is specific to the natural innovation representation. Using the classical Doeblin decomposition, we briefly explain how every uniformly ergodic Markov chain admits a causal Bernoulli shift representation with geometrically decaying functional dependence coefficients after enlarging the underlying probability space. The paper concludes with some remarks and open problems.

\section{Pathological examples}\label{2}

We begin by establishing a simple obstruction to the existence of a causal Bernoulli shift representation on the natural innovation space. As we shall see, such a representation necessarily implies a pullback synchronization property for the associated synchronous coupling. This criterion will then be used to construct simple uniformly ergodic autoregressive models for which the natural innovation representation cannot be used.

From now on, for integers $s\leq t$, we write
\[
f_s^t=f_t\circ f_{t-1}\circ\cdots\circ f_s
\]
for the composition of the random mappings.

Note that in this case, for $t>0$, the equality $X_t=f_1^t(X_0)$ implies a representation of the form $$X_t=H_t\left(\varepsilon_t,\ldots,\varepsilon_1,X_0\right)$$
where $H_t$ is a suitable measurable function.

\begin{proposition}\label{bern}
Let
\[
X_t=F(X_{t-1},\varepsilon_t)=f_t(X_{t-1})
\]
be a stationary Markov chain with invariant probability distribution $\pi$. Suppose that $(X_t)_{t\in\mathbb Z}$ admits a causal Bernoulli shift representation on the natural innovation space. Then the associated synchronous coupling satisfies the following pullback synchronization property.
$$\lim_{n\rightarrow \infty}\int\int \min\left\{\left\vert f_{-n}^0(x)-f_{-n}^0(y)\right\vert,1\right\}\pi(dx)\pi(dy)=0\mbox{ a.s.}.$$
\end{proposition}

\paragraph{Proof.}

Suppose that for $t\in\Z$
\[
X_t
=
H(\varepsilon_t,\varepsilon_{t-1},\ldots)
\]
for some measurable function $H$.
Let also $\phi:\R\rightarrow \R$ be a measurable and bounded function.
Setting $\mathcal{F}_{-n}^0=\sigma\left(\varepsilon_0,\ldots,\varepsilon_{-n}\right)$, the Doob martingale limit theorem ensures that 
$$\lim_{n\rightarrow \infty}\E\left[\phi(X_0)\vert\mathcal{F}_{-n}^0\right]=\phi(X_0)\mbox{ a.s.}$$
Note that we have the equality 
$$\E\left[\phi(X_0)\vert\mathcal{F}_{-n}^0\right]=\int \phi\left(f_{-n}^0(x)\right) \pi(dx)\mbox{ a.s.}$$
One can deduce that almost surely, for all continuous and bounded mapping $\phi$, 
$$\lim_{n\rightarrow \infty}\int \phi\left(f_{-n}^0(x)\right) \pi(dx)=\phi(X_0).$$
Taking $\phi(z)=\min\left\{d(z,X_0),1\right\}$, we have from the triangular inequality
$$\int\int \min\left\{\left\vert f_{-n}^0(x)-f_{-n}^0(y)\right\vert,1\right\}\pi(dx)\pi(dy)\leq 2\int \phi\left(f_{-n}^0(x)\right)\pi(dx)\rightarrow 0\mbox{ a.s.}$$
which shows the result.$\square$

\subsection{A discontinuous autoregression}

Consider the autoregressive model

\[
F(x,\varepsilon)=r(x)+\varepsilon,
\]

where the innovation $\varepsilon$ admits a probability density which is positive and bounded away from zero on every compact interval.

Partition the real line into intervals of equal length and define
\[
r(x)
=
h\sum_{\ell\in\mathbb Z}
\mathbf 1_{[a+2\ell h,a+(2\ell+1)h)}(x),
\]
where $a\in\mathbb R$ and $h>0$.

Writing
\[
I_p=[a+ph,a+(p+1)h),
\]
the function $r$ alternates between the values $0$ and $h$ on successive intervals.

Suppose that $(x,y)\in I_p\times I_{p'}$, where $p$ and $p'$ have opposite parity.
Then
\[
r(x)-r(y)=\pm h,
\]
and therefore
\[
|f_1(x)-f_1(y)|=h.
\]
Since the intervals have length $h$, the points $f_1(x)$ and $f_1(y)$ belong to two consecutive intervals, and hence again to intervals of opposite parity. By induction,
\[
|f_1^n(x)-f_1^n(y)|=h,\qquad n\ge1.
\]
Applying the same argument to the compositions
$f_{-n}^0=f_0\circ f_{-1}\circ\cdots\circ f_{-n}$ yields
\[
|f_{-n}^0(x)-f_{-n}^0(y)|=h,\qquad n\ge1.
\]
Consequently, the synchronous coupling never synchronizes.

Since the Markov chain is uniformly ergodic and admits a unique invariant probability measure $\pi$ equivalent to the Lebesgue measure and which has a full support, the set $J\times J'$ with $J=\cup_{p\in\Z}I_{2p}$ and $J'=\cup_{p\in\Z} I_{2p+1}$ has positive $\pi\otimes \pi$ probability and Proposition~\ref{bern} implies that no causal Bernoulli shift representation exists on the natural innovation space.

\subsection{A second discontinuous example}

Let $\beta>0$ and define

\[
f_t(x)
=
\begin{cases}
\varepsilon_t,&x>0,\\
-\beta\varepsilon_t,&x\le0.
\end{cases}
\]

If the $\varepsilon_t'$s do not vanish a.s. and $x$ and $y$ have opposite signs, then so do
$f_t(x)$ and $f_t(y)$ for every realization of the noise.

Hence,

\[
|f_1^n(x)-f_1^n(y)|
=
(1+\beta)|\varepsilon_n|,
\qquad n\ge1,
\]

and

\[
|f_{-n}^0(x)-f_{-n}^0(y)|
=
(1+\beta)|\varepsilon_0|,
\qquad n\ge1.
\]

Therefore the pullback synchronous coupling never synchronizes.

Assuming again that $\varepsilon$ has a probability density which is positive and bounded away from zero on compact intervals, the associated Markov chain is uniformly ergodic. Proposition~\ref{bern} then shows that no causal Bernoulli shift representation exists on the natural innovation space.

\section{Smooth models and synchronous couplings}\label{3}

\subsection{Forward synchronous couplings and Lyapunov exponents}

Proposition~\ref{bern} is based on a pullback construction obtained by replacing the remote past in a Bernoulli shift representation. By contrast, the functional dependence coefficients introduced by \cite{WBW2005} naturally involve a \emph{forward} synchronous coupling. Indeed, for two initial conditions $x,y\in\R$, let
\[
X_n(x)=g_n(x),\qquad
X_n(y)=g_n(y),
\]
where
\[
g_n=f_1^n=f_n\circ f_{n-1}\circ\cdots\circ f_1.
\]
The two trajectories are therefore driven by the same realization of the innovation sequence, and the question is whether the initial perturbation is eventually forgotten, that is,
\[
|g_n(x)-g_n(y)|\longrightarrow0.
\]
Under stationarity, the pullback coupling appearing in Proposition~\ref{bern} and this forward synchronous coupling are closely related. The latter is also the natural framework for introducing Lyapunov exponents.

Throughout this section we assume that the random mappings are continuously differentiable. The (top) Lyapunov exponent is defined by
\[
\lambda(x)
=
\lim_{n\rightarrow\infty}
\frac1n
\log |g_n'(x)|,
\]
whenever the limit exists almost surely.
 We refer to \cite{Arnold1998} for a general introduction to Lyapunov exponents in random dynamical systems.

Assume now that the Markov chain admits a unique invariant probability measure $\pi$ and that
\[
\int
\E\!\left|
\log |f_1'(x)|
\right|
\,\pi(dx)
<\infty.
\]
Then,
\[
\frac1n
\log |g_n'(x)|
=
\frac1n
\sum_{t=1}^n
\log
|f_t'(X_{t-1}(x))|
\longrightarrow
\int
\E
\log
|f_1'(y)|
\,\pi(dy),
\]
almost surely by the ergodic theorem. Accordingly, we define the Lyapunov exponent of the stationary Markov chain by
\begin{equation}
\label{def}
\lambda
=
\int
\E
\log |f_1'(y)|
\,\pi(dy).
\end{equation}

For iterated random maps of the form
\[
f_t(x)=F(x,\varepsilon_t),
\]
it is tempting to apply the functional dependence measure directly on the natural innovation space generated by $(\varepsilon_t)_{t\in\mathbb Z}$. In this framework, the decay of the functional dependence coefficients is naturally related to the behavior of the forward synchronous coupling, that is, to the question whether two trajectories driven by the same innovation sequence eventually forget their initial separation.

Synchronization under common noise has been extensively studied in the random dynamical systems literature; see for instance
\cite{flandoli2017synchronization1},
\cite{flandoli2017synchronization2},
\cite{gess2024lyapunov},
\cite{scheutzow2016synchronization},
and
\cite{newman2018necessary}.
It is natural to expect that the sign of the Lyapunov exponent determines the behavior of the synchronous coupling: a negative exponent should lead to synchronization of nearby trajectories, whereas a positive exponent should prevent the forgetting of initial perturbations. Although this heuristic is often correct under suitable structural assumptions, the relationship is considerably more subtle in general. For instance, \cite{scheutzow2016synchronization} exhibits simple examples for which synchronization occurs despite a positive Lyapunov exponent, and conversely where synchronization fails although the Lyapunov exponent is negative.

Most available synchronization results for iterated random maps on the real line rely on additional structural assumptions, such as monotonicity, order preservation, compactness of the state space, or global average contractivity. In particular, the average contractivity assumptions considered by Letac \cite{Letac}, Diaconis and Freedman \cite{diaconis1999iterated}, and Wu and Shao \cite{WuShao2004} imply the negativity of the Lyapunov exponent. The converse, however, is false in general.

The next subsection shows that, for the family of smooth autoregressive models considered in this paper, the behavior of the natural synchronous coupling undergoes a qualitative transition which coincides with a change in the sign of the Lyapunov exponent.
\subsection{Functional AR(1) with positive Lyapunov exponent}
We consider the model
\begin{equation}
    \label{ar1}
    X_t=r(X_{t-1})+\varepsilon_t,
    \end{equation}
with $(\varepsilon_t)_{t\in\N}$ a sequence of i.i.d. random variables with a density $f_{\varepsilon}$ positive everywhere and lower-bounded on any compact interval.
The measurable mapping $r$ is assumed to be bounded. The current model was studied in \cite{DG80} and a simulation study was even provided in \cite{doukhan83} in order to consider sensitivity results wrt to small values of the underlying noise. The study was based on one example in \cite{ruelle77}. 
It is widely known that under these assumptions, the Markov chain $(X_t)_{t\in\N}$ satisfies the Doeblin condition and it is uniformally ergodic and $\phi-$mixing with an exponential rate.
See for example 
 \cite{kulik} for the connexions to $\beta-$mixing  and $\varphi-$mixing properties. In particular, there exists a unique invariant probability measure $\pi$. The measure $\pi$ has a probability density $f_{\pi}$ w.r.t. the Lebesque measure and satisfying the identity
$$
f_{\pi}(y)=\int f_{\varepsilon}(y-r(x))f_{\pi}(x)dx.$$

Note that $f_{\pi}$ is positive.
Now let $X_0$ be a random variable with probability distribution $\pi$ and independent from the sequence $\left(\varepsilon_t\right)_{t\geq 1}$. Let also $\varepsilon'_1$ a copy of $\varepsilon_1$ and independent of $\sigma\left(X_0,\varepsilon_t:t\geq 1\right)$.
We introduce a version $(X'_t)_{t\geq 0}$ of $(X_t)_{t\geq 0}$ such that 
$$
\begin{array}{lll}
X'_0&=&X_0,\quad\\
X_1'&=&r(X_0)+\varepsilon_1',\quad \\
X'_t&=&r(X_{t-1}')+\varepsilon_t,\qquad\mbox{ for } t\geq 2.
\end{array}$$  
Finally let 
\begin{equation}
\label{D_t}
    D_t=X_t-X'_t,\quad t\geq 1.
    \end{equation}
In what follows, $m_2$ will denote the Lebesgue measure on $\R^2$ and $m$ the Lebesgue measure on $\R$.
\begin{theorem}\label{p1}
Assume that $r$ is continuously differentiable with a bounded and uniformly continuous derivative $r'$ and such that
\[
\lambda:=\E\log |r'(X_0)|>0,\qquad \E\log_-|r'(X_0)|<\infty.
\]
For $t\geq 1$, set,
\[
Q_t=\frac{|r(X_t)-r(X_t')|}{|X_t-X_t'|}\,\mathbf{1}_{(D_t\ne0)}.
\]

Suppose also that, for some $\delta_0>0$,
\begin{equation}\label{eq:new-log-condition}
\lim_{\eta\downarrow0}\sup_{t\geq 2}
\E\left[
\log_- Q_t\,
\mathds{1}_{\{|r'(X_t)|\leq \eta,\ |D_t|\leq \delta_0\}}
\right]=0.
\end{equation}
Then $\P\left(\displaystyle\lim_{n\rightarrow \infty}D_n=0\right)=0$.
Additionally, for $m_2-$almost every $(x,y)\in\R^2$, 
$$\P\left(\lim_{n\rightarrow \infty}\left\vert f_1^n(x)-f_1^n(y)\right\vert=0\right)=0.$$

\end{theorem}

\begin{remark}
Theorem \ref{p1} shows that when the Lyapunov exponent $\lambda$ is positive, one cannot expect forward synchronization of the paths. However one cannot use the converse of Proposition \ref{bern} to show that a Bernoulli shift representation is not possible, since this result only concerns the pullback synchronization.  
But from the Borel-Cantelli lemma, $\sum_{n\geq 1}\Vert D_n\Vert_q$ cannot be finite which shows that any Bernoulli shift representation on the natural innovations space would not satisfy the usual weak dependence condition, since in this case $\delta_q(n)=\Vert D_{n+1}\Vert_q$.
\end{remark}

\begin{proof}
We first record that $D_t\neq0$ a.s. for every $t\geq1$. Indeed, the positivity of the Lyapunov exponent and the positivity of the invariant density imply
\[
m\{x:r'(x)=0\}=0.
\]
Consequently, $r$ has the following inverse-null-set property: if $A$ is negligible for the Lebesgue measure, then $r^{-1}(A)$ is negligible. This follows by covering compact subsets of $\{r'\neq0\}$ by finitely many intervals on which $r$ is a $C^1$-diffeomorphism, and then using that $m\{r'=0\}=0$.

We next prove that, if $(X,Y)$ is absolutely continuous with respect to the Lebesgue measure on $\mathbb R^2$, then so is $(r(X),r(Y))$. Let $B\subset\mathbb R^2$ be a Borel set such that $m_2(B)=0$, and set
\[
A=\{(x,y):(r(x),r(y))\in B\}.
\]
By Fubini's theorem, for almost every $v$, the section $B_v=\{u:(u,v)\in B\}$ is negligible. Using the inverse-null-set property of $r$, we get
\[
m\left(r^{-1}(B_{r(y)})\right)=0
\]
for almost every $y$, and therefore $m_2(A)=0$. Hence $(r(X),r(Y))$ is absolutely continuous whenever $(X,Y)$ is.

Since $(X_1,X_1')$ is absolutely continuous, an induction gives that $(X_t,X_t')$ is absolutely continuous for every $t\geq1$. In particular,
\[
\P(D_t=0)=0,
\qquad t\geq1.
\]
Thus $Q_t$ is well-defined a.s. for every $t\geq1$, and
\begin{equation}\label{eq:Dt-product}
|D_{n+1}|=|D_1|\prod_{t=1}^n Q_t,
\qquad n\geq1.
\end{equation}

We now prove that, on the event $\{D_t\to0\}$,
\begin{equation}\label{eq:cesaro-adjacent-general}
\frac1n\sum_{t=1}^n\left(
\log Q_t-\log |r'(X_t)|
\right)\longrightarrow0
\quad\mbox{in probability}.
\end{equation}
Let
\[
\Delta_t=\log Q_t-\log |r'(X_t)|.
\]
Fix $\eta>0$. Since $r'$ is uniformly continuous, with modulus of continuity $\omega$, if $|D_t|\leq\delta$ and $|r'(X_t)|>\eta$, then
\[
\left|\frac{r(X_t)-r(X_t')}{X_t-X_t'}-r'(X_t)\right|\leq \omega(\delta).
\]
For $\delta\in (0,\delta_0)$, we get
\[
|\Delta_t|\mathds{1}_{\{|r'(X_t)|>\eta,\ |D_t|\leq\delta\}}
\leq C_\eta\omega(\delta),
\]
for a constant $C_\eta$ depending only on $\eta$.

On the complementary set $\{|r'(X_t)|\leq\eta\}$, we use
\[
|\Delta_t|
\leq \log_+Q_t+\log_-Q_t+
\log_+|r'(X_t)|+
\log_-|r'(X_t)|.
\]
Since $r'$ is bounded, $Q_t\leq\|r'\|_\infty$ and both positive parts are bounded. Moreover, by stationarity of $(X_t)$ and the integrability of $\log_-|r'(X_0)|$,
\[
\E\left[
\log_-|r'(X_t)|\mathds{1}_{\{|r'(X_t)|\leq\eta\}}
\right]
=
\E\left[
\log_-|r'(X_0)|\mathds{1}_{\{|r'(X_0)|\leq\eta\}}
\right]
\longrightarrow0.
\]
Together with \eqref{eq:new-log-condition}, this gives
\begin{equation}\label{eq:bad-zone-control}
\lim_{\eta\downarrow0}\sup_{t\geq2}
\E\left[
|\Delta_t|\mathds{1}_{\{|r'(X_t)|\leq\eta,\ |D_t|\leq\delta_0\}}
\right]=0.
\end{equation}

Let $C=\{D_t\to0\}$. On $C$, for the above fixed $\delta$, the set of times for which $|D_t|>\delta$ is a.s. finite. Hence the contribution of these times to Cesaro averages is negligible. Combining the previous estimates and then applying Markov's inequality, we obtain
\[
\frac1n\sum_{t=1}^n |\Delta_t|\mathds{1}_C
\longrightarrow0
\quad\mbox{in probability}.
\]
This proves \eqref{eq:cesaro-adjacent-general} on $C$.

Assume now, by contradiction, that $\P(C)>0$. From \eqref{eq:cesaro-adjacent-general}, there exists a subsequence $(n_j)$ such that
\[
\frac1{n_j}\sum_{t=1}^{n_j}\left(
\log Q_t-\log |r'(X_t)|
\right)\mathds{1}_C
\longrightarrow0
\quad\mbox{a.s.}
\]
On the other hand, by the ergodic theorem,
\[
\frac1n\sum_{t=1}^n\log |r'(X_t)|
\longrightarrow \lambda>0
\quad\mbox{a.s.}
\]
Therefore, on $C$,
\[
\frac1{n_j}\sum_{t=1}^{n_j}\log Q_t
\longrightarrow \lambda>0
\quad\mbox{a.s.}
\]
In particular,
\[
\sum_{t=1}^{n_j}\log Q_t\longrightarrow +\infty
\quad\mbox{on }C.
\]
Using \eqref{eq:Dt-product}, we get
\[
|D_{n_j+1}|\longrightarrow +\infty
\quad\mbox{on }C,
\]
which contradicts the definition of $C$. Thus $\P\left(\displaystyle\lim_{t\rightarrow \infty} D_t =0\right)=0$.
Since 
$$\P\left(\displaystyle\lim_{t\rightarrow \infty} D_t =0\right)=\int\int \P\left(\lim_{t\rightarrow \infty}\left\vert f_2^t(x)-f_2^t(y)\right\vert=0\right)f_{X_1,X_1'}(x,y)dxdy,$$
where 
$$f_{X_1,X_1'}(x,y)=\E\left[f_{\varepsilon}\left(x-r(X_0)\right)f_{\varepsilon}(y-r(X_0))\right]>0,$$ 
the absence of forward synchronization follows.
\end{proof}
\subsection{An explicit example}

We consider the Markov chain
\[
X_t=af(X_{t-1})+\varepsilon_t,
\]
where $(\varepsilon_t)_{t\geq 1}$ are i.i.d. real-valued random variables with density $f_\varepsilon$ and $f:\R\rightarrow \R$ is a continuously differentiable periodic function with period $T>0$. Without loss of generality, we assume that $\Vert f'\Vert_{\infty}=1$.
Assume that $f_\varepsilon$ satisfies the following periodized boundedness condition:
\begin{equation}\label{eq:periodized-noise-bound}
M:=\sup_{u\in\R}\sum_{\ell\in\Z} f_\varepsilon(u+\ell T)<\infty.
\end{equation}
This assumption is satisfied, for instance, by the standard Gaussian density.

Let $\pi_a$ be an invariant probability measure of the chain, and assume that under $\pi_a$ the stationary random variable $X$ has density $p_a$. The Lyapunov exponent associated with the synchronous coupling is now denoted by
\[
\lambda_a=\int_\R \log |a f'(x)|\,\pi_a(d x).
\]
Observe that when $\vert a\vert<1$, then $r=a f$ is a contraction and a Bernoulli shift expansion is available 
in term of the innovations sequence. Moreover the functional dependence measure coefficients decay geometrically (when the required $q$th moment of the noise $\varepsilon_0$ is finite). In this case, the Lyapunov exponent $\lambda_a$ is negative. In the following result, we consider situations for which $\vert a\vert>1$.

\begin{proposition}\label{lyapsin}

Suppose that there exist some $\delta_0>0$ such that
\begin{equation}\label{integral}
\int_0^T\log_{-}\vert f'(u)\vert du<\infty,\quad \lim_{\eta\searrow 0}\sup_{d:\vert d\vert\leq \delta_0}\int_0^T\mathds{1}_{\vert r'(u)\vert\leq \eta}\log_{-}\frac{\vert r(u+d)-r(u)\vert}{\vert d\vert} du=0.
\end{equation}
Under \eqref{eq:periodized-noise-bound}, there exists $k_0<\infty$ such that, for every $\vert a\vert >k_0$,
\[
\lambda_a>0.
\]
More precisely,
\[
\lambda_a\geq \log \vert a\vert-C_\varepsilon,
\]
where
\[
C_\varepsilon=M\int_0^T\bigl|\log|f'(u)|\bigr|\,d u<\infty.
\]
Moreover the conclusions of Theorem \ref{p1} are valid.
\end{proposition}

\begin{proof}
Since the chain is stationary, one may write
\[
X\stackrel{d}{=}a f(Y)+\varepsilon,
\]
where $Y\sim\pi_a$ and $\varepsilon$ is independent of $Y$. Hence the density $p_a$ of $X$ is of the form
\[
p_a(x)=\int_\R f_\varepsilon(x-a f(y))\,\pi_a(d y).
\]
Equivalently, if $\mu_a$ denotes the law of $a f(Y)$, then
\[
p_a(x)=\int_\R f_\varepsilon(x-z)\,\mu_a(d z).
\]
Define the $T$-periodization of $p_a$ by
\[
\overline p_a(u)=\sum_{\ell\in\Z} p_a(u+T\ell),\qquad u\in[0,T).
\]
By Tonelli's theorem,
\begin{align*}
\overline p_a(u)
&=\sum_{\ell\in\Z}\int_\R f_\varepsilon(u+T\ell-z)\,\mu_a(d z)\\
&=\int_\R \sum_{\ell\in\Z} f_\varepsilon(u-z+T\ell)\,\mu_a(d z)\\
&\leq M.
\end{align*}
The constant $M$ is independent of $a$ and $u$.

Now use the $T$-periodicity of $x\mapsto\log|f'(x)|$. Since $p_a$ is a probability density,
\begin{align*}
\int_\R \log|f'(x)|\,p_a(x)d x
&=\int_0^T \log|f'(u)|\,\overline p_a(u)d u.
\end{align*}
From our assumptions, the function $u\mapsto \log| f'(u)|$ is integrable on $[0,T]$. Therefore,
\begin{align*}
\int_\R \log|f'(x)|\,p_a(x)d x
&\geq -\int_0^T \bigl|\log|f'(u)|\bigr|\overline p_a(u)d u\\
&\geq -M\int_0^T \bigl|\log|f'(u)|\bigr|d u\\
&=:-C_\varepsilon.
\end{align*}
Consequently,
\begin{align*}
\lambda_a
&=\int_\R \log|af'(x)|\,p_a(x)d x\\
&=\log \vert a\vert +\int_\R \log|f'(x)|\,p_a(x)d x\\
&\geq \log \vert a\vert-C_\varepsilon.
\end{align*}
Thus $\lambda_a>0$ as soon as $\vert a\vert >\exp(C_\varepsilon)$.

In order to apply Theorem \ref{p1}, it remains to check the integrability condition \ref{eq:new-log-condition}. Setting $a_t=r(X_{t-1})$, observing that $D_t=a_t-r(X'_{t-1})$ and using the periodicity of $(r,r')$, we have the bounds:
\begin{eqnarray*}
\E\left[\mathds{1}_{\vert D_t\vert\leq \delta_0}\mathds{1}_{\vert r'(X_t)\vert \leq \eta}\log_{-}Q_t\right]&\leq& \E\left[\mathds{1}_{\vert D_t\vert\leq \delta_0}\int_{\R}\mathds{1}_{\vert r'(u)\vert\leq \eta}\log_{-}\frac{\vert r(u+D_t)-r(u)\vert}{\vert D_t\vert}f_{\varepsilon}(u-a_t)du\right] \\
&\leq& M \sup_{\vert d\vert\leq \delta_0}\int_0^T\mathds{1}_{\vert r'(u)\vert\leq \eta}\log_{-}\frac{\vert r(u+d)-r(u)\vert}{\vert d\vert} du,
\end{eqnarray*}
which shows that (\ref{eq:new-log-condition}) is satisfied and completes the proof.
\end{proof}

\begin{remark}
Condition (\ref{integral}) is not always easy to check for an arbitrary mapping $f$. 
Let us check it for the sine and cosine function. Here $T=2\pi$.
For $r(x)=a\sin x$, we have
$$\frac{r(u+d)-r(u)}{d}=a\cos(u+d/2)\frac{\sin(d/2)}{d/2}.$$

If $\vert d\vert\leq \delta_0$ sufficiently small, then
$$\frac{\vert r(u+d)-r(u)\vert}{\vert d\vert}\geq \vert a\cos(u+d/2)\vert/2.$$
Setting $a_\eta:=\eta/|a|$, 
Let
\[
g(u)=\log_-\left(a|\cos u|/2\right),
\qquad
E_{\eta}=\{u\in[0,T): |\cos u|\leq a_{\eta}\}.
\]
 we get
\begin{eqnarray*}
\int_0^T\mathds{1}_{\vert r'(u)\vert\leq \eta}\log_{-}\frac{\vert r(u+d)-r(u)\vert}{\vert d\vert} du&\leq& \int_0^T\mathds{1}_{\vert\cos(u)\vert\leq a_{\eta}} g(u+d/2)du\\
&\leq & \sup_{\vert c\vert \leq \delta_0/2}\int_{E_{\eta}}g(v+c)dv.
\end{eqnarray*}
The last term tends to $0$ as $\eta\downarrow0$. Indeed, $g\in L^1([-\delta_0/2,T+\delta_0/2])$ and if $m$ denotes the Lebesgue measure on $\R$, $m(E_{\eta})\to 0$ as $\eta\downarrow0$, and by absolute continuity of the Lebesgue integral,
\[
\sup_{\vert c\vert\leq \delta_0/2}\int_{E_{\eta}} g(v+c)dv
=
\sup_{\vert c\vert\leq \delta_0/2}\int_{E_{\eta}+c}g(w)dw
\leq
\sup_{A\in \mathcal{B}([-\delta_0/2,T+\delta_0/2]):m(A)\leq m(E_{\eta})}\int_A g(w)dw
\longrightarrow0.
\]
For $r(x)=a\cos(x)$, the decomposition 
$$\frac{r(u+d)-r(u)}{d}=-a\sin(u+d/2)\frac{\sin(d/2)}{d/2}$$
leads to similar result.
\end{remark}

\begin{remark}\label{gausssin}
For the standard Gaussian density
\[
f_\varepsilon(x)=\frac{1}{\sqrt{2\pi}}e^{-x^2/2},
\]
the bound \eqref{eq:periodized-noise-bound} is immediate. Indeed, the function
\[
u\mapsto \sum_{\ell\in\Z}f_\varepsilon(u+T\ell)
\]
is $T$-periodic. For $u\in[-T/2,T/2]$,
\[
|u+T\ell|\geq T|\ell|-T/2,
\]
and therefore
\[
\sum_{\ell\in\Z} f_\varepsilon(u+T\ell)
\leq
\frac{1}{\sqrt{2\pi}}+\frac{2}{\sqrt{2\pi}}
\sum_{\ell\geq1}
\exp\left(-\frac{(T\ell-T/2)^2}{2}\right)<\infty.
\]
\end{remark}

\begin{remark}
When $f(x)=\sin(x)$ and $\varepsilon_0$ follow the standard Gaussian distribution, we found numerically that $k_0\approx 5.827$. This shows that this threshold is much larger than $1$ but the condition $\lambda_a$ might be valid for smaller values of $\vert a\vert$.

\end{remark}

\section{Bernoulli representations for Doeblin Markov chains}\label{4}

The examples of Sections~\ref{2} and~\ref{3} show that the natural innovation representation of an autoregressive Markov chain need not be suitable for the direct application of functional dependence techniques. This should not be interpreted as the non-existence of a causal Bernoulli shift representation. At the contrary, uniformly ergodic Markov chains always admit such a representation after enlarging the underlying probability space. First, uniform ergodicity always entails the Doeblin property given below, see for instance \cite{MeynTweedie1993}, Theorem $16.0.2$. We recall that 
a transition kernel $P$ satisfies a $m-$step Doeblin condition if there exist $\alpha>0$, a positive integer $m$ and a probability measure $\nu$ such that
\[
P^m(x,A)\geq \alpha \nu(A),
\qquad
x\in E,\;A\in\mathcal E.
\]
Such a condition is satisfied for all the Markov chains examples considered previously with $m=1$. We now only consider the case $m=1$ but a similar construction using blocks is possible for any value of $m$.
Writing
\[
P(x,\cdot)
=
\alpha \nu(\cdot)
+
(1-\alpha)Q(x,\cdot),
\]
for a suitable Markov kernel $Q$ and using the Nummelin splitting construction together with the standard quantile representation of Markov kernels on R, we may realize the chain recursively from three independent i.i.d.\ sequences:

\begin{itemize}
\item Bernoulli random variables $(B_n)$ with
$\mathbb P(B_n=1)=\alpha$;
\item random variables $(Y_n)$ with distribution $\nu$;
\item uniform random variables $(U_n)$ used to simulate the residual kernel $Q$.
\end{itemize}

More precisely, there exists a measurable function $\Phi$ such that
\[
X_n
=
\Phi(X_{n-1},B_n,Y_n,U_n),
\]
with
\[
X_n=
\begin{cases}
Y_n,&B_n=1,\\
Q^{-1}(X_{n-1},U_n),&B_n=0,
\end{cases}
\]
where $Q^{-1}$ denotes any measurable simulation map associated with the kernel $Q$.

The Bernoulli variables play the role of regeneration indicators. Whenever $B_n=1$, the chain is restarted from the distribution $\nu$ and completely forgets its previous state. Consequently, two copies driven by the same innovation sequence necessarily coalesce after the first regeneration time. Since the latter has a geometric distribution, the associated functional dependence coefficients decay geometrically. More precisely, If $T$ denotes the coupling time for two paths generated synchronously except at time $0$ and if $X_0\in \L^r$ for $r>q$, we have from Holder inequality,
$$\delta_q(n)\leq C\P\left(T>n\right)^{1/q-1/r}\leq C\left(1-\alpha\right)^{n(1/q-1/r)}.$$

One can then deduce a Bernoulli shift representation with exponentially decaying functional dependence coefficients. The examples of the previous sections therefore illustrate that the main issue is not the existence of a Bernoulli representation itself, but rather whether the \emph{natural} innovation representation enjoys the required forgetting properties. Once the previous representation defined, the functional AR representation holds true setting $\varepsilon_t=X_t-r(X_{t-1})$. We also point out that such Markov chains are also automatically $\phi-$mixing with geometric decays and then many statistical applications can be deduced.

\section{Conclusion}\label{conclusion}

The examples presented in this note illustrate that, for Markov chains generated by iterated random maps, the natural innovation representation may not be suitable for a direct application of the functional dependence measures, even in case the Markov chain is uniformly ergodic. From this perspective, functional dependence should be viewed as a property of a particular stochastic representation rather than of the process distribution itself.

Several questions remain open. First, our smooth example suggests a close connection between the behavior of the natural synchronous coupling and the sign of an associated Lyapunov exponent. While a positive exponent prevents the forgetting of initial perturbations in our setting, it would be of considerable interest to determine whether a negative Lyapunov exponent is sufficient to recover the natural Bernoulli shift representation beyond the globally contractive framework of \cite{diaconis1999iterated} or the geometric moment contraction condition of \cite{WuShao2004}.

A second direction concerns nonlinear autoregressive models whose ergodicity is established through Foster--Lyapunov drift conditions and minorization arguments. Such models fall well outside the globally contractive setting and include, for instance, threshold autoregressive processes and related nonlinear Markov chains. Whether the natural innovation representation is compatible with functional dependence techniques in this broader context appears to be an interesting open problem.
\normalfont
\normalsize
\bibliographystyle{plainnat}

\bibliography{hist}
\end{document}